 \documentclass{amsart}%
\usepackage[letterpaper,top=4cm,bottom=4cm,left=2.5cm,right=2.5cm,marginparwidth=1.75cm]{geometry}

 \usepackage{amssymb}
 \usepackage{amsfonts}
 \usepackage{amsmath}
 \usepackage{graphicx}
 \usepackage{hyperref}%
 \usepackage{setspace}
 \usepackage{stmaryrd}
\usepackage[shortlabels]{enumitem}
\usepackage{tikz}
\usepackage{tikz-cd}
\usetikzlibrary{matrix,arrows.meta}

 \providecommand{\U}[1]{\protect\rule{.1in}{.1in}}
 \newtheorem{theorem}{Theorem}[section]
 \theoremstyle{plain}

 \newtheorem{corollary}{Corollary}[section]

 \newtheorem{lemma}{Lemma}[section]

 \newtheorem{proposition}{Proposition}[section]

 \numberwithin{equation}{section}

\theoremstyle{definition}
 \newtheorem{example}{Example}[section]
 \newtheorem{remark}{Remark}[section]

\DeclareMathOperator{\jac}{Jac}

\DeclareMathOperator{\spec}{Spec}
\DeclareMathOperator{\depth}{depth}
\DeclareMathOperator{\ext}{Ext}
\DeclareMathOperator{\Hom}{Hom}
\DeclareMathOperator{\Ann}{Ann}
\DeclareMathOperator{\height}{ht}

\begin{document}

\title[Properties of Bi-Amalgamated Algebras]{Cohen--Macaulay and Gorenstein Properties of Bi-Amalgamated Algebras with Applications to Algebroid Curves}

\author{Efe G\"urel}
\address{T\"UB\.{I}TAK Natural Sciences High School, 41400 Kocaeli, Turkey}
\email{efegurel54@gmail.com}

\author{Abuzer G\"und\"uz}
\address{Department of Mathematics, Sakarya University, 54050 Sakarya, Turkey}
\email{abuzergunduz@sakarya.edu.tr}
\thanks{Corresponding author: Abuzer G\"und\"uz.}
\thanks{The second author was supported by a T\"UB\.{I}TAK 2219 grant, titled ``Grassmann Varieties and Singularity Theory.''}

\subjclass[2020]{13H10; 13B02; 13C15; 14H20}
\keywords{Bi--amalgamated algebra, amalgamated algebra, Cohen--Macaulay ring, Gorenstein ring, algebroid curve, universally catenary ring}

\begin{abstract}
Let $A \bowtie^{f,g} (J,J')$ be the bi--amalgamation of a commutative ring $A$ with $(B,C)$ along the ideals $(J,J')$ with respect to the ring homomorphisms $(f,g)$. In this article, we study the basic homological properties of the bi--amalgamated algebra construction. We first calculate the dimension and depth of the bi--amalgamated algebra under fairly general circumstances and derive necessary and sufficient conditions for Cohen--Macaulayness in terms of maximal and big Cohen--Macaulay modules of $A$. Furthermore, we characterize the Gorenstein property of the bi--amalgamated algebra through the canonical modules of $f(A)+J$ and $g(A)+J'$. We apply our results to the theory of curve singularities by constructing Gorenstein algebroid curves through bi--amalgamated and amalgamated algebras. We also give a brief remark concerning the universally catenary property of $A\bowtie^{f,g}(J,J')$.
\end{abstract}

\maketitle

\section{Introduction}

Let $A$ and $B$ be commutative rings, let $J$ be an ideal of $B$, and let $f\colon A \to B$ be a ring homomorphism. The set
\[
A \bowtie^{f} J = \{(a, f(a)+j) \mid a \in A,\ j \in J\}
\]
is called the \emph{amalgamation} of $A$ and $B$ along $J$ with respect to $f$, and it is a subring of $A \times B$. This construction extends the notion of amalgamated duplication introduced by D'Anna and Fontana in \cite{DF}. A more general framework for amalgamations was later developed in \cite{DFF} and \cite{DFF2}, where amalgamated algebras were constructed via pull--back. This perspective allowed the authors to establish several transfer results for algebraic and homological properties from the rings $A$ and $f(A)+J$ to the amalgamated algebra $A \bowtie^{f} J$. Since then, amalgamations have attracted significant attention and have been successfully applied in various contexts; see, for instance, \cite{DMZ} and \cite{MMZ2}. \medskip

Let $A$, $B$, and $C$ be commutative rings, let $J$ and $J'$ be ideals of $B$ and $C$, respectively, and let $f\colon A \to B$ and $g\colon A \to C$ be ring homomorphisms satisfying
\[
f^{-1}(J)=g^{-1}(J').
\]
The set
\[
A \bowtie^{f,g}(J,J')=\{(f(a)+j, g(a)+j') \mid a\in A,\ (j,j')\in J\times J'\}
\]
is called the \emph{bi--amalgamation} of $A$ with $(B,C)$ along $(J,J')$ with respect to $(f,g)$. It is a subring of $B \times C$. Bi--amalgamations, which naturally generalize amalgamations and duplications, were introduced in \cite{KLT}. On top of classical pull--back constructions, the rings $f(A)+J$, $A+J$, and furthermore the CPI--extension $C(A,I)$ can be viewed as bi--amalgamations. \medskip

Subsequent works have focused on the transfer of specific algebraic properties to bi--amalgamations. The coherence property was studied in \cite{OM}. In \cite{KMM}, necessary and sufficient conditions were established for a bi--amalgamation to be arithmetical, with applications to weak global dimension and the transfer of the semihereditary property. The transfer of Gaussian and Pr\"ufer properties was investigated in \cite{MM}, where the authors recovered known results for amalgamations and constructed new examples of rings satisfying these properties. For further background on bi--amalgamated algebras, we refer to \cite{OM}, \cite{KLT}, \cite{KMM}, and \cite{MM}. \medskip

The study of Cohen--Macaulay and Gorenstein properties in this context originates from the amalgamated duplication $A \bowtie I$ in the works of D'Anna \cite{Gorenstein}. It was shown that, under suitable annihilator conditions, $A \bowtie I$ is Gorenstein if and only if $A$ admits a canonical ideal isomorphic to $I$, and later generalizations characterized the quasi--Gorenstein and normal cases \cite{quasigorenstein,SS}. Furthermore, an application to singularity theory has been proposed, where amalgamated duplication was used as a tool to construct Gorenstein algebroid curves. Such algebroid curves are characterized through the symmetry properties of their value semigroups. The value semigroups of duplicated monomial algebroid branches are also constructed explicitly in \cite{Gorenstein}.\medskip

Investigations of Cohen--Macaulay and Gorenstein properties were generalized to the framework of amalgamations in \cite{iranlilar,DFF3}. Their main results include the calculation of the depth and dimension of amalgamation rings explicitly under suitable conditions, and the characterization of the Cohen--Macaulay and Gorenstein properties, along with a brief discussion of universally catenary rings. Our article will be primarily concerned with these results and their extensions to the bi--amalgamated algebra. They further studied embedding dimensions, generalized Cohen--Macaulay rings and modules, the Serre condition, and quasi-- and generalized Gorenstein rings in relation to amalgamations. We have not deemed such investigations necessary for this paper and have kept them as open questions for further research. There have been no investigations regarding the construction of algebroid curves using amalgamated algebras in the literature to the best of our knowledge.\medskip

In Section 2, we calculate the dimension and depth of the bi--amalgamated algebra by using tools from local cohomology in Theorem \ref{dimdepthcalclemma}. In particular, we prove in Theorem \ref{cohenmacthm} that under suitable assumptions, $A\bowtie^{f,g}(J,J')$ is Cohen--Macaulay if and only if $A$ is Cohen--Macaulay with maximal Cohen--Macaulay modules $J, J'$.\medskip

In Section 3, we first prove a $\Hom$-lemma useful in our investigations in Lemma \ref{AnnLemma}. Then, we prove a necessary condition for $A\bowtie^{f,g}(J,J')$ to be Gorenstein in Proposition \ref{gorprop}. Finally, we completely characterize the Gorenstein property of $A\bowtie^{f,g}(J,J')$ in Theorem \ref{gorensteintheorem} by using the pull--back description of bi--amalgamations under suitable assumptions. \medskip

In Section 4, we apply our results to the theory of curve singularities. We show that bi--amalgamation constructions applied to algebroid curves also yield algebroid curves in Theorem \ref{algebroidthm}. Furthermore, we determine when the bi--amalgamated algebroid curve is Gorenstein in Theorem \ref{algebroidgorthm}, utilizing our findings from the previous section. This gives us a way of constructing algebroid Gorenstein curves, which are of particular importance.\medskip

In Section 5, we give a brief remark on how the bi--amalgamated algebra is a homomorphic image of a polynomial ring over $A$ in Proposition \ref{homimageprop}. As a result, we characterize the universally catenary property of $A\bowtie^{f,g}(J,J')$ in Proposition \ref{catenaryprop}. \medskip

As previously mentioned, the bi--amalgamated algebra generalizes the classical pull--back constructions in the literature. We shall give a few examples of interest, which can be found in \cite{KLT}. The amalgamated algebra can be obtained by setting $I=f^{-1}(J)$ and $\iota=\text{id}_A$, yielding
\[
A\bowtie^{f}J\cong A\bowtie^{\iota,f}(I,J).
\]
Upon further specialization to $f=\text{id}_A$, we get the amalgamated duplication as
\[
A\bowtie I\cong A\bowtie^{\iota,\iota}(I,I)
\]
for an ideal $I$ of $A$. In the case that $I^2=0$, the amalgamated duplication agrees with the idealization $A\ltimes I$. The ring $f(A)+J$ can be obtained as
\[
f(A)+J \cong A\bowtie^{\pi,f}(0,J)
\]
with the projection $\pi:A\to A/I$. Taking $f=i$ to be an embedding of $A$ into $B$, we obtain the ring $A+J$ as
\[
A+J \cong A\bowtie^{\pi,i}(0,J).
\]
This shows that the classical $A+XB[X]$ and $D+M$ pull--back constructions can be seen as bi--amalgamations. Finally, as a non--classical example, the Boisen--Sheldon CPI--extension \cite{CPI} is obtained as follows. For an ideal $I$ of $A$, the sets $\overline{S}=(A/I)\setminus Z(A/I)$ and $S=\left\{s\in A \mid \pi(s) \in \overline{S} \right\}$ are multiplicatively closed subsets of $A/I$ and $A$, respectively. Define the canonical ring homomorphisms $\varphi:S^{-1}A \to \overline{S}^{-1}(A/I)$ and $f:A\to S^{-1}A$. The CPI--extension of $A$ with respect to $I$ is given by
\[
C(A,I)=\varphi^{-1}\left(A/I\right)=f(A)+S^{-1}I
\]
as a subring of $S^{-1}A$. It can be obtained as the bi--amalgamation 
\[
C(A,I)=A\bowtie^{\pi,f}\left(0,S^{-1}I\right).
\]
The fundamental properties of the bi--amalgamated algebra can be found in \cite{KLT}. We have collected the main properties we will use in the following proposition for a convenient reference.  

\begin{proposition} \label{Prop.introd.}
\begin{enumerate}[(1)]
    \item (\cite[Proposition 4.1]{KLT}) \label{isomlemma}
    We have $A\bowtie^{f,g}(J,J')/(0\times J')\cong f(A)+J$ and $A\bowtie^{f,g}(J,J')/(J\times 0)\cong g(A)+J'$. Furthermore, for an ideal $I$ of $A$, we have $A\bowtie^{f,g}(J,J')/(I\bowtie^{f,g}(J,J'))\cong A/(I+I_0)$. 
    
    \item (\cite[Proposition 5.3]{KLT}) \label{primelemma}
    Let us define associated prime ideals of $A\bowtie^{f,g}(J,J')$ for given $L\in\spec(f(A)+J)$ and $L'\in \spec(g(A)+J')$ as
    \begin{align*}
        \overline{L} &= (L\times(g(A)+J'))\cap\left(A\bowtie^{f,g} (J,J')\right) \\
        &= \{(f(a)+j,g(a)+j') \mid a\in A,\ (j,j')\in J\times J',\ f(a)+j\in L\}, \\
        \overline{L'} &= ((f(A)+J)\times L')\cap\left(A\bowtie^{f,g}(J,J')\right) \\
        &= \{(f(a)+j,g(a)+j') \mid a\in A,\ (j,j')\in J\times J',\ g(a)+j'\in L'\}.
    \end{align*}
    Let $\mathfrak{P}\in \spec\left(A\bowtie^{f,g}(J,J')\right)$. Then:
    \begin{enumerate}[(a)]
        \item We have $J\times J'\subseteq \mathfrak{P}$ if and only if there exists a unique $\mathfrak{p}\supseteq I_0$ in $\spec(A)$ such that $\mathfrak{P}=\mathfrak{p}\bowtie^{f,g}(J,J')$. In this case, there exist $L\supseteq J$ in $\spec(f(A)+J)$ and $L'\supseteq J'$ in $\spec(g(A)+J')$ such that $\mathfrak{P}=\overline{L}=\overline{L'}$.
        
        \item We have $J\times J'\nsubseteq \mathfrak{P}$ if and only if there exists a unique $L\in\spec(f(A)+J)$ or $L'\in\spec(g(A)+J')$ such that $J\nsubseteq L$ or $J'\nsubseteq L'$ with $\mathfrak{P}=\overline{L}$ or $\mathfrak{P}=\overline{L'}$. Respectively, $\left(A\bowtie^{f,g}(J,J')\right)_{\overline{L}}=(f(A)+J)_{L}$ or $\left(A\bowtie^{f,g}(J,J')\right)_{\overline{L'}}=(g(A)+J')_{L'}$.
    \end{enumerate}    
    Consequently,
    \[
    \spec\left(A\bowtie^{f,g}(J,J')\right)=\{\overline{L} \mid L\in \spec(f(A)+J)\cup\spec(g(A)+J')\}.
    \]
    
    \item (\cite[Proposition 5.4]{KLT}) \label{locallemma}
    $A\bowtie^{f,g}(J,J')$ is local if and only if $J\neq B$ and both $f(A)+J$ and $g(A)+J'$ are local. Further, the maximal ideal of $A\bowtie^{f,g}(J,J')$ has the form $\mathfrak{m}\bowtie^{f,g}(J,J')$, where $\mathfrak{m}$ is the unique maximal ideal of $A$ containing $I_0$. If $A$ is local, then $A\bowtie^{f,g}(J,J')$ is local if and only if $J \times J'\subseteq\jac(B\times C)$.
\end{enumerate}
\end{proposition}

\section{Cohen--Macaulay Property of $A\bowtie^{f,g}(J,J')$}

In this section, we characterize the Cohen--Macaulay property of the ring $A\bowtie^{f,g}(J,J')$. To do so, we first calculate the dimension and the depth of $A\bowtie^{f,g}(J,J')$ under fairly general circumstances. It was first shown in \cite{Gorenstein} that for a Cohen--Macaulay ring $A$, the ring $A\bowtie I$ is Cohen--Macaulay if and only if $I$ is a maximal Cohen--Macaulay $A$--module. This easily follows from the $A$--module isomorphism $A\bowtie I\cong A\oplus I$. When $J$ is a finitely generated $A$--module, it is shown in \cite{DFF3} that under the condition $J\subseteq\jac B$, the ring $A\bowtie^{f}J$ is Cohen--Macaulay if and only if $J$ is a maximal Cohen--Macaulay $A$--module. This also follows from the $A$--module isomorphism $A\bowtie^{f}J\cong A\oplus J$. It is also remarked that it is considerably more difficult to extend these results to the case where $J$ is not finitely generated. This case was thoroughly discussed in \cite{iranlilar}. The main difficulty in establishing analogous results for the bi--amalgamated algebra is the fact that an analogous isomorphism of $A$--modules does not always hold. This problem first led us to investigate when such an isomorphism exists and adopt methods similar to that of \cite{iranlilar}.\medskip

We first introduce some necessary definitions. The exposition here is largely based on \cite{broadmann, brunsherzog}, where all the quoted facts can be found. For convenience, we suppress the underlying rings while discussing the dimensions and depths of modules. This is possible in all of our applications by \cite[Exercise 1.2.26]{brunsherzog}. Let $A$ be a Noetherian commutative ring with an ideal $\mathfrak{a}$ and an $A$--module $M$. The $\mathfrak{a}$--torsion functor $\Gamma_{\mathfrak{a}}$ is defined as 
\begin{align*}
    \Gamma_{\mathfrak{a}}(M)=\bigcup_{n=0}^{\infty}\left( 0:_M \mathfrak{a}^n\right),
\end{align*}
which is the set of all elements of $M$ which are annihilated by some power of $\mathfrak{a}$. Clearly, the $\mathfrak{a}$-torsion functor is naturally equivalent to the functor $\displaystyle    \varinjlim_{n\in\mathbb N}\Hom_A(A/\mathfrak{a}^n,{}_\bullet).
$ The local cohomology functors $H^i_{\mathfrak{a}}$ are the right derived functors of $\Gamma_{\mathfrak{a}}$. The modules
\begin{align*}
    H^i_{\mathfrak{a}}(M)\cong\varinjlim_{n\in\mathbb N}\ext^i_A(A/\mathfrak{a}^n,M)\cong\varinjlim_{n\in\mathbb N}H^i(\Hom_A(A/\mathfrak{a}^n,I^\bullet))\cong H^i\left( \varinjlim_{n\in\mathbb N}\Hom_A(A/\mathfrak{a}^n,I^\bullet) \right)
\end{align*}
are called the local cohomology modules of $M$ with respect to $\mathfrak{a}$, where $I^\bullet$ is an injective resolution of $M$. From these definitions it follows that for any two ideals $\mathfrak{a}$ and $\mathfrak{b}$, we have $\Gamma_{\mathfrak{a}}=\Gamma_{\mathfrak{b}}$ precisely when $\sqrt{\mathfrak{a}}=\sqrt{\mathfrak{b}}$. Therefore, if $\sqrt{\mathfrak{a}}=\sqrt{\mathfrak{b}}$, $H^i_{\mathfrak{a}}(M)\cong H^i_{\mathfrak{b}}(M)$ for all $A$--modules $M$. For a given ring homomorphism $f:A\to B$, let $|_A:B\text{-\textbf{Mod}}\to A\text{-\textbf{Mod}}$ be the functor of restriction of scalars by $f$. Through the independence theorem, the functors $H^i_{\mathfrak{a}B}({}_\bullet)|_A$ and $H^i_{\mathfrak{a}}({}_\bullet|_A)$ are naturally equivalent, resulting in an $A$--module isomorphism $H^i_{\mathfrak{a}B}(M)\cong H^i_{\mathfrak{a}}(M)$ for all $B$--modules $M$. Grothendieck's vanishing theorem states that $H^i_{\mathfrak{a}}(M)=0$ for all $i>\dim M$. From now on, let us assume that $(A,\mathfrak{m})$ is a Noetherian local ring. Grothendieck's non--vanishing theorem states that for a non--zero finitely generated $A$--module $M$, we have $H^{\dim M}_{\mathfrak{m}}(M)\neq0$. Depth of an $A$--module is intimately connected with the non--vanishing of the local cohomology modules and can be described as
\begin{align*}
    \depth M=\inf\left( i:H^{i}_{\mathfrak{m}}(M)\neq0 \right).
\end{align*}
In the case where $M$ is finitely generated, we obtain the classical depth given by the length of a maximal $M$--sequence in $\mathfrak{m}$.  It follows from the above discussion that a finitely generated $A$--module $M$ over a Noetherian local ring $(A,\mathfrak{m})$ is Cohen--Macaulay if and only if $H^i_{\mathfrak{m}}(M)=0$ for all $i<\dim M$. Furthermore, it is a well-known fact that $\depth M\le \dim A$ whenever the left-hand side is finite, which is the case when $\mathfrak{m}M\neq M$. If $\depth M=\dim A$ for a finitely generated $A$--module $M$, we say that $M$ is a maximal Cohen--Macaulay $A$--module. In the case where $M$ is not necessarily finitely generated, we say that $M$ is a big Cohen--Macaulay $A$--module. \medskip

It is well known (see \cite[Proposition 4.1]{KLT}) that the sequence
\begin{equation}\label{exactseq}
    \begin{tikzcd}
	0 & {J\oplus J'} & {A\bowtie^{f,g}(J,J')} & {A/I_0} & 0
	\arrow[from=1-1, to=1-2]
	\arrow["i", from=1-2, to=1-3]
	\arrow["\pi", from=1-3, to=1-4]
	\arrow[from=1-4, to=1-5]
\end{tikzcd}
\end{equation}
is exact, where $i$ and $\pi$ denote the natural embedding and projection maps, respectively. Our first result is to characterize when the exact sequence (\ref{exactseq}) splits.

\begin{lemma}
    The exact sequence 
    \[\begin{tikzcd}
	0 & {J\oplus J'} & {A\bowtie^{f,g}(J,J')} & {A/I_0} & 0
	\arrow[from=1-1, to=1-2]
	\arrow["i", from=1-2, to=1-3]
	\arrow["\pi", from=1-3, to=1-4]
	\arrow[from=1-4, to=1-5]
\end{tikzcd}\]
of $A$--modules splits if and only if there exist $x\in J, y\in J'$ such that $f(i)=xf(i)$ and $g(i)=yg(i)$ for all $i\in I_0$.
\end{lemma}
\begin{proof}
    By the splitting lemma, the exact sequence splits if and only if there exists a homomorphism $\varphi:A/I_0\to A\bowtie^{f,g}(J,J')$ such that $\pi \circ \varphi=\operatorname{id}_{A/I_0}$ where $\pi:A\bowtie^{f,g}(J,J')\to A/I_0$ with $\pi: (f(a)+j,g(a)+j')\mapsto a+I_0$. Such a homomorphism is uniquely determined by $\varphi(1+I_0)=(1-x,1-y)$ since we have $\varphi(a+I_0)=(f(a)(1-x),g(a)(1-y))$. If there exist $x\in J$ and $y\in J'$ such that $f(i)=xf(i)$ and $g(i)=yg(i)$ for all $i\in I_0$, it can be seen that $\varphi$ indeed satisfies $\pi \circ \varphi=\operatorname{id}_{A/I_0}$. Now let $\pi \circ \varphi=\operatorname{id}_{A/I_0}$. $\pi \circ \varphi(1+I_0)=1+I_0$ shows that $x\in J$ and $y\in J'$. Take an arbitrary $i\in I_0$. Then $\varphi(i+I_0)=0$ shows that $f(i)=xf(i)$ and $g(i)=yg(i)$ for all $i\in I_0$. This concludes the proof.
\end{proof}

\begin{corollary}
    If $J\subseteq \jac(B)$ and $J'\subseteq \jac(C)$, then the above sequence splits if and only if $I_0= \ker f\cap\ker g$.
\end{corollary}
\begin{proof}
    Since $x\in J\subseteq \jac (B)$ and $y\in J'\subseteq\jac(C)$, the elements $1-x$ and $1-y$ are invertible and thus the condition is equivalent to $f(i)=0$ and $g(i)=0$ for all $i\in I_0$. It is obvious that $\ker f\cap \ker g\subseteq I_0$.
\end{proof}

In order to utilize the independence theorem of local cohomology, we first determine the maximal ideal of the local ring $A\bowtie^{f,g}(J,J')$, following \cite{iranlilar}.

\begin{lemma}\label{radicallemma}
    Let $(A,\mathfrak{m})$ be a local ring with $J\subseteq \jac(B)$ and $J'\subseteq \jac(C)$. Assume that $f^{-1}(\mathfrak{q}) \neq \mathfrak{m}$ and $g^{-1}(\mathfrak{q}') \neq \mathfrak{m}$ for all $\mathfrak{q}\in \spec (B)\setminus V(J)$ and $\mathfrak{q}'\in\spec (C)\setminus V(J')$, respectively. Then $\mathfrak{m}\bowtie^{f,g}(J,J')=\sqrt{\mathfrak{m}A\bowtie^{f,g}(J,J')}$.
\end{lemma}
\begin{proof}
    Our assumptions and Proposition \ref{Prop.introd.}(3) guarantee that $A\bowtie^{f,g}(J,J')$ is a local ring with the maximal ideal $\mathfrak{m}\bowtie^{f,g}(J,J')$. The maximality of $\mathfrak{m}\bowtie^{f,g}(J,J')$ gives us $\sqrt{\mathfrak{m}A\bowtie^{f,g}(J,J')}\subseteq \mathfrak{m}\bowtie^{f,g}(J,J')$. Conversely, let us take a prime ideal $\mathfrak{P}$ such that $\mathfrak{m}A\bowtie^{f,g}(J,J')\subseteq \mathfrak{P}$. Proposition \ref{Prop.introd.} \ref{primelemma} gives us two cases. In the first case $\mathfrak{P}=\mathfrak{p}\bowtie^{f,g}(J,J')$ for some prime ideal $\mathfrak{p}\supseteq I_0$ of $A$. This implies $\mathfrak{m}\subseteq \mathfrak{p}$, and thus $\mathfrak{p}=\mathfrak{m}$ by the maximality of $\mathfrak{m}$. In the other case, $\mathfrak{P}=\overline{L}$ for some $L\in(\spec(B)\cup\spec(C))\setminus(V(J)\cup V(J'))$. This however contradicts the assumption that for all $\mathfrak{q}\in \spec (B)\setminus V(J)$ and $\mathfrak{q}'\in\spec (C)\setminus V(J')$, we have $f^{-1}(\mathfrak{q})\neq \mathfrak{m}$ and $g^{-1}(\mathfrak{q}')\neq \mathfrak{m}$. Therefore the only prime ideal of $A\bowtie^{f,g}(J,J')$ which contains $\mathfrak{m}A\bowtie^{f,g}(J,J')$ is $\mathfrak{m}\bowtie^{f,g}(J,J')$ and $\mathfrak{m}\bowtie^{f,g}(J,J')=\sqrt{\mathfrak{m}A\bowtie^{f,g}(J,J')}$.
\end{proof}

The assumption that $J$ and $J'$ are finitely generated $A$--modules guarantees that  for all $\mathfrak{q}\in \spec (B)\setminus V(J)$ and $\mathfrak{q}'\in\spec (C)\setminus V(J')$, $f^{-1}(\mathfrak{q})$ and $g^{-1}(\mathfrak{q}')$ are not equal to $\mathfrak{m}$. Thus we obtain the following corollary.

\begin{corollary}
        Let $(A,\mathfrak{m})$ be a local ring with $J\subseteq \jac(B)$ and $J'\subseteq \jac(C)$ finitely generated $A$--modules. Then $\mathfrak{m}\bowtie^{f,g}(J,J')=\sqrt{\mathfrak{m}A\bowtie^{f,g}(J,J')}$.
\end{corollary}

We are now able to calculate the dimension and the depth of the bi--amalgamated algebra $A\bowtie^{f,g}(J,J')$.

\begin{theorem}\label{dimdepthcalclemma}
    Let $(A,\mathfrak{m})$ be a local ring with $J\subseteq \jac(B)$ and $J'\subseteq \jac(C)$. Then we have the following.
    \begin{enumerate}
        \item Let $J$ and $J'$ be finitely generated $A$--modules. Then 
        \begin{align*}
            \dim A\bowtie^{f,g}(J,J')=\dim A/(\ker f\cap \ker g).
        \end{align*}
        Furthermore,
        \begin{align*}
            \depth A\bowtie^{f,g}(J,J')=\min(\depth A/I_0,\depth (J\oplus J'))
        \end{align*}
        except possibly when $\depth A/I_0+1=\depth (J\oplus J')$ and the connecting homomorphism 
    \begin{align*}
        \partial^{\depth A/I_0}:H_{\mathfrak{m}}^{\depth A/I_0}(A/I_0)\to H_{\mathfrak{m}}^{\depth A/I_0+1}(J\oplus J')
    \end{align*}
    obtained by applying the functor $H_{\mathfrak{m}}$ to (\ref{exactseq}) has a trivial kernel, in which case $\depth A\bowtie^{f,g}(J,J')=\depth A/I_0+1$.
        \item Let $I_0= \ker f\cap\ker g$. Assume that $f^{-1}(\mathfrak{q}) \neq \mathfrak{m}$ and $g^{-1}(\mathfrak{q}') \neq \mathfrak{m}$ for all $\mathfrak{q}\in \spec (B)\setminus V(J)$ and $\mathfrak{q}'\in\spec (C)\setminus V(J')$, respectively. Then 
        \begin{align*}
            \dim A\bowtie^{f,g}(J,J')= \dim A/(\ker f\cap \ker g)
        \end{align*}
        and
        \begin{align*}
            \depth A\bowtie^{f,g}(J,J')=\min( \depth A/(\ker f\cap \ker g),\depth (J\oplus J')).
        \end{align*}
    \end{enumerate}
\end{theorem}
\begin{proof}
    (1) The map $\iota:A\to A\bowtie^{f,g}(J,J')$ with $\iota:a\mapsto(f(a),g(a))$ gives us an embedding of $A/(\ker f\cap \ker g)$ into $A\bowtie^{f,g}(J,J')$. Since $J$ and $J'$ are finitely generated $A$--modules, they are also finitely generated $A/(\ker f\cap \ker g)$--modules and this makes $A\bowtie^{f,g}(J,J')$ a finitely generated $A/(\ker f\cap \ker g)$--module as well. Thus $A\bowtie^{f,g}(J,J')$ is an integral extension of $A/(\ker f\cap \ker g)$ and they have the same dimension. The depth inequality is a simple consequence of \cite[Exercise 1.2.26]{brunsherzog} and \cite[Proposition 1.2.9]{brunsherzog} applied to the exact sequence (\ref{exactseq}). \medskip

        (2) Let us also note that $J$ and $J'$ are naturally $A/(\ker f\cap \ker g)$--modules. Our assumptions with Proposition \ref{Prop.introd.}(3) and Lemma \ref{radicallemma} guarantee that $ A\bowtie^{f,g}(J,J')$ is a local ring with the maximal ideal $\mathfrak{m}\bowtie^{f,g}(J,J')=\sqrt{\mathfrak{m}A\bowtie^{f,g}(J,J')}$. Therefore, by the radical invariance and the independence theorem of local cohomology, we get
        \begin{align*}
            H^i_{\mathfrak{m}\bowtie^{f,g}(J,J')}\left(A\bowtie^{f,g}(J,J')\right)\cong H^i_{\mathfrak{m} A\bowtie^{f,g}(J,J')}\left(A\bowtie^{f,g}(J,J')\right)\cong H^i_{\mathfrak{m}}\left(A\bowtie^{f,g}(J,J')\right)
        \end{align*}
        for all $i\ge 0$. By the preceding corollary, the exact sequence (\ref{exactseq}) splits. Thus, applying the local cohomology functor yields
        \begin{align}\label{maincohomisom}
            H^i_{\mathfrak{m}\bowtie^{f,g}(J,J')}\left(A\bowtie^{f,g}(J,J')\right)\cong H^i_{\mathfrak{m}}(A/I_0)\oplus H^i_{\mathfrak{m}}(J\oplus J').
        \end{align}
        Let us take $i=\dim A/I_0$. By Grothendieck's non--vanishing theorem, $H^{\dim A/I_0}_{\mathfrak{m}}(A/I_0)\neq 0$ and thus
        \begin{align*}
            H^{\dim A/I_0}_{\mathfrak{m}\bowtie^{f,g}(J,J')}\left(A\bowtie^{f,g}(J,J')\right)\neq 0.
        \end{align*} 
        Grothendieck's vanishing theorem now gives us $\dim A/I_0\le \dim A\bowtie^{f,g}(J,J')$. Assume for the sake of contradiction that $\dim A\bowtie^{f,g}(J,J')>\dim A/I_0$. Taking $i=\dim A\bowtie^{f,g}(J,J')$ we have \begin{align*}
            H^{\dim A\bowtie^{f,g}(J,J')}_{\mathfrak{m}\bowtie^{f,g}(J,J')}\left(A\bowtie^{f,g}(J,J')\right)\neq 0.
        \end{align*}
        However this contradicts Grothendieck's vanishing theorem as we have
        \begin{align*}
            H^{\dim A\bowtie^{f,g}(J,J')}_{\mathfrak{m}}(A/I_0)=H^{\dim A\bowtie^{f,g}(J,J')}_{\mathfrak{m}}(J\oplus J')=0.
        \end{align*}
        Therefore we must have
        \begin{align*}
            \dim A\bowtie^{f,g}(J,J')=\dim A/I_0=\dim A/(\ker f\cap \ker g).
        \end{align*}
        Now let us show the equality of depths. Taking $i<\depth A\bowtie^{f,g}(J,J')$, we obtain $H^i_{\mathfrak{m}}(A/I_0)= H^i_{\mathfrak{m}}(J\oplus J')=0$ and therefore
        \begin{align*}
            \depth A\bowtie^{f,g}(J,J')\le \min( \depth A/I_0,\depth (J\oplus J')).
        \end{align*}
        Taking $i=\depth A\bowtie^{f,g}(J,J')$, we obtain 
        \begin{align*}
            H^{\depth A\bowtie^{f,g}(J,J')}_{\mathfrak{m}}(A/I_0)\oplus H^{\depth A\bowtie^{f,g}(J,J')}_{\mathfrak{m}}(J\oplus J')\cong H^{\depth A\bowtie^{f,g}(J,J')}_{\mathfrak{m}\bowtie^{f,g}(J,J')}\left(A\bowtie^{f,g}(J,J')\right)\neq 0.
        \end{align*}
        Hence we have $H^{\depth A\bowtie^{f,g}(J,J')}_{\mathfrak{m}}(A/I_0)\neq 0$ or $H^{\depth A\bowtie^{f,g}(J,J')}_{\mathfrak{m}}(J\oplus J')\neq 0$. In either case, we get
        \begin{align*}
            \min( \depth A/I_0,\depth (J\oplus J'))\le \depth A\bowtie^{f,g}(J,J')
        \end{align*}
        which implies
        \begin{align*}
            \depth A\bowtie^{f,g}(J,J')=\min( \depth A/(\ker f\cap \ker g),\depth (J\oplus J')).
        \end{align*}
        This concludes the proof.
\end{proof}

\begin{remark}
    Note that in the first part of Theorem \ref{dimdepthcalclemma}, the first case is typically true. One may also deduce this fact by the long exact sequence of local cohomology obtained via (\ref{exactseq}) instead of using \cite[Proposition 1.2.9]{brunsherzog}. These approaches utilizing the $\ext$ and local cohomology functors are essentially equivalent.
\end{remark}

The dimension and depth calculations of amalgamations, amalgamated duplications and idealizations now easily follow.

\begin{corollary}
    Let $(A,\mathfrak{m})$ be a local ring with $J\subseteq \jac(B)$ a finitely generated $A$--module. Then $\dim A\bowtie^f J=\dim A$ and $\depth A\bowtie^f J=\min(\depth A,\depth J)$.
\end{corollary}

\begin{corollary}
    Let $(A,\mathfrak{m})$ be a local ring with $M$ a finitely generated $A$--module. Then $\dim A\ltimes M=\dim A$ and $\depth A\ltimes M=\min(\depth A,\depth M)$.
\end{corollary}

\begin{corollary}
    Let $(A,\mathfrak{m})$ be a local ring and $I$ an ideal of $A$.  Then $\dim A\bowtie I=\dim A$ and $\depth A\bowtie I=\min(\depth A,\depth I)$.
\end{corollary}

Through the dimension and depth calculation, we can easily characterize when the bi--amalgamated algebra $A\bowtie^{f,g}(J,J')$ is Cohen--Macaulay as follows.

\begin{theorem}\label{cohenmacthm}
    Let $(A,\mathfrak{m})$ be a local ring with $J\subseteq \jac(B)$ and $J'\subseteq \jac(C)$. Then we have the following.
    \begin{enumerate}
        \item Let $J$ and $J'$ be finitely generated $A$--modules. Assume $\height(\ker f\cap \ker g)=0$. Then, $A\bowtie^{f,g}(J,J')$ is Cohen--Macaulay if and only if $A$ is Cohen--Macaulay, $J$ and $J'$ are maximal Cohen--Macaulay modules with $\depth A/I_0=\depth A$ in the usual case and $\depth A/I_0=\depth A-1$ in the exceptional case (where the exceptional case refers to the condition detailed in Theorem \ref{dimdepthcalclemma}(1)).

        \item  Let $I_0= \ker f\cap\ker g$. Assume that $f^{-1}(\mathfrak{q}) \neq \mathfrak{m}$ and $g^{-1}(\mathfrak{q}') \neq \mathfrak{m}$ for all $\mathfrak{q}\in \spec (B)\setminus V(J)$ and $\mathfrak{q}'\in\spec (C)\setminus V(J')$, respectively. Then, if $A\bowtie^{f,g}(J,J')$ is Cohen--Macaulay, it follows that $A/(\ker f\cap \ker g)$ is Cohen--Macaulay. If furthermore $\depth J,\depth J'<\infty$, then $A\bowtie^{f,g}(J,J')$ is Cohen--Macaulay if and only if $A/(\ker f\cap \ker g)$ is Cohen--Macaulay, and $J$ and $J'$ are big Cohen--Macaulay $A/(\ker f\cap \ker g)$--modules. 
    \end{enumerate}
\end{theorem}
\begin{proof}
     (1) In this case, since $A$ is Cohen--Macaulay and $\height(\ker f \cap \ker g) = 0$, we have $\dim A/(\ker f \cap \ker g) = \dim A$. Thus, by Theorem \ref{dimdepthcalclemma}, $\dim A\bowtie^{f,g}(J,J')=\dim A$. The exceptional case is trivial, so let us assume the usual case. Assume that $A\bowtie^{f,g}(J,J')$ is Cohen--Macaulay, so we must have $\depth A\bowtie^{f,g}(J,J')=\dim A$. First notice that
        \begin{align*}
            \depth A&\ge\depth A/I_0\\
&\ge\min(\depth A/I_0 ,\depth (J\oplus J'))\\&=\dim A.
        \end{align*}
    Since the inequality $\dim A\ge \depth A$ always holds,  we get $\dim A=\depth A$ and $A$ is Cohen--Macaulay. We have the chain of inequalities
    \begin{align*}
        \dim A&=\depth A\bowtie^{f,g}(J,J')\\
        &=\min(\depth A/I_0 ,\depth (J\oplus J'))\\&\le\depth J\\&\le\dim A
    \end{align*}
    where the last inequality follows from the fact that $\depth J<\infty$. It must be the case that equality holds in every step of this chain and we get $\depth J=\dim A$ so $J$ is a maximal Cohen--Macaulay module. Similarly, $J'$ is a maximal Cohen--Macaulay module and $\depth A/I_0=\depth A$. Conversely, it is easy to see that $A\bowtie^{f,g}(J,J')$ is Cohen--Macaulay under the given conditions by Theorem \ref{dimdepthcalclemma}.\medskip

    (2) Let $A\bowtie^{f,g}(J,J')$ be Cohen--Macaulay. We will prove that $A/(\ker f\cap \ker g)$ is also Cohen--Macaulay. Using Theorem \ref{dimdepthcalclemma}, we obtain
        \begin{align*}
            \dim A/(\ker f\cap\ker g)&=\dim A\bowtie^{f,g}(J,J')\\&=\depth A\bowtie^{f,g}(J,J')\\
&=\min(\depth A/(\ker f\cap\ker g) ,\depth (J\oplus J'))\\&\le\depth A/(\ker f\cap\ker g)\\&\le\dim A/(\ker f\cap\ker g).
        \end{align*}
        Similarly, equality must hold in every step of this chain. Therefore we get that $A/(\ker f\cap\ker g)$ is Cohen--Macaulay. Let us assume further that $\depth J,\depth J'<\infty$. An analogous argument shows that $J$ and $J'$ are big Cohen--Macaulay $A/(\ker f\cap\ker g)$--modules. It is easy to see conversely that these assumptions imply $A\bowtie^{f,g}(J,J')$ is Cohen--Macaulay.
\end{proof}

As a consequence, we reobtain the characterization of the Cohen--Macaulay property for the amalgamated algebra, amalgamated duplication and idealization as in \cite{DFF3} and \cite{Gorenstein}.

\begin{corollary}
    Let $(A,\mathfrak{m})$ be a local ring with $J\subseteq \jac(B)$ a finitely generated $A$--module. Then $A\bowtie^f J$ is Cohen--Macaulay if and only if $A$ is Cohen--Macaulay and $J$ is a maximal Cohen--Macaulay $A$--module.
\end{corollary}
\begin{corollary}
    Let $(A,\mathfrak{m})$ be a local ring with $M$ a finitely generated $A$--module. Then $A\ltimes M$ is Cohen--Macaulay if and only if $A$ is Cohen--Macaulay and $M$ is a maximal Cohen--Macaulay $A$--module.
\end{corollary}
\begin{corollary}
     Let $(A,\mathfrak{m})$ be a local ring and $I$ an ideal of $A$. Then $A\bowtie I$ is Cohen--Macaulay if and only if $A$ is Cohen--Macaulay and $I$ is a maximal Cohen--Macaulay $A$--module.
\end{corollary}

\section{Gorenstein Property of $A\bowtie^{f,g}(J,J')$}

The aim of this section is to give some necessary and sufficient conditions for the Gorenstein property of the bi--amalgamated algebra $A\bowtie^{f,g}(J,J')$. The Gorenstein property for amalgamated duplication was characterized by D'Anna in \cite{Gorenstein}. Under the condition that $A$ is a local Cohen--Macaulay ring with a proper ideal $I$, $A\bowtie I$ is Gorenstein if and only if $A$ has a canonical ideal which is isomorphic to $I$. Shapiro \cite{Shapiro} pointed out and corrected an implicit assumption of $\Ann_A(I)=0$ in the original proof, which, for Noetherian rings, is equivalent to $I$ being a regular ideal. This investigation was further generalized to the quasi--Gorenstein property in \cite{quasigorenstein}. The Gorenstein property for the amalgamated algebra was characterized using the pull--back construction in \cite{DFF3}. A study on quasi--Gorensteinness of amalgamated algebras can be found in \cite{iranlilar}, where results of \cite{DFF3} were further improved. Here we use an approach similar to that of \cite{DFF3} and use the pull--back description of $A\bowtie^{f,g}(J,J')$ to characterize the Gorenstein property.\medskip

We now outline the definition of a canonical module and a Gorenstein ring, following the classical terminology in \cite{herzogkunz,brunsherzog}. We call an $A$--module $\omega$ a canonical module of the local ring $(A,\mathfrak{m})$ if
\begin{align*}
    \omega\otimes _A \widehat{A}\cong\Hom_A\left( H_{\mathfrak{m}}^{\dim A}(A),E_A(A/\mathfrak{m}) \right),
\end{align*}
where as usual $E_A$ means the injective hull over $A$. If such a canonical module exists, then it is unique up to isomorphism and is finitely generated as an $A$--module with dimension $\dim A$. $A$ is called quasi--Gorenstein if $H_{\mathfrak{m}}^{\dim A}(A)\cong E_A(A/\mathfrak{m})$. This is equivalent to the existence of a canonical module $\omega_A$ of $A$ that is a free $A$--module of rank one \cite{somebasic}. A quasi--Gorenstein and Cohen--Macaulay ring is called Gorenstein. A Noetherian local ring $A$ admits a canonical module if and only if it is a homomorphic image of a Gorenstein ring. In fact, if $A$ is a homomorphic image of a local Gorenstein ring $B$, then one may construct $\omega_A$ as
\begin{align*}
    \omega_A\cong\ext_B^{\dim B-\dim A}(A,B).
\end{align*}

In order to start our investigation, we first need the following lemma.

\begin{lemma}\label{AnnLemma}
    If $\Ann_{g(A)+J'}(J')=0$, then
\begin{align*}
    \Hom_{A\bowtie^{f,g}(J,J')}\left(f(A)+J,A\bowtie^{f,g}(J,J')\right)\cong J\times 0
\end{align*}
holds as $f(A)+J$--modules.    
\end{lemma}
\begin{proof}
    By Proposition \ref{Prop.introd.} \ref{isomlemma}, we have the following
    \begin{align*}
        \Hom_{A\bowtie^{f,g}(J,J')}\left(f(A)+J,A\bowtie^{f,g}(J,J')\right)&\cong \Hom_{A\bowtie^{f,g}(J,J')}\left(A\bowtie^{f,g}(J,J')/0\times J',A\bowtie^{f,g}(J,J')\right)\\
&\cong \Ann_{A\bowtie^{f,g}(J,J')}(0\times J')\\
&=\{(f(a)+j,g(a)+j'):g(a)+j'\in \Ann_{g(A)+J'}(J') \}\\&=\{(f(a)+j,0):g(a)\in J'\}.
    \end{align*}
    Furthermore, since $g(a)\in J'$, we get $a\in g^{-1}(J')=I_0=f^{-1}(J)$. This implies that $f(a) \in J$, and consequently the first coordinate $f(a)+j$ entirely belongs to $J$. Thus, the module is exactly $J\times 0$.
\end{proof}
Lemma \ref{AnnLemma} reduces to the classical constructions of \cite{Gorenstein,DFF3,iranlilar} in special cases. To avoid confusion with the future injectivity assumption on $f$, we use the homomorphism $g$ and the ideal $J'$ for the amalgamated algebra in this section.
\begin{corollary}
        If $\Ann_{g(A)+J'}(J')=0$, then
\begin{align*}
    \Hom_{A\bowtie^{g}J'}\left(A,A\bowtie^{g}J'\right)\cong g^{-1}(J')\times 0
\end{align*}
holds as $A$--modules.    
\end{corollary}
\begin{corollary}
        If $\Ann_{A}(M)=0$, then
\begin{align*}
    \Hom_{A\ltimes M}(A,A\ltimes M)\cong M\times 0
\end{align*}
holds as $A$--modules.    
\end{corollary}
\begin{corollary}
        If $\Ann_{A}(I)=0$, then
\begin{align*}
    \Hom_{A\bowtie I}(A,A\bowtie I)\cong I\times 0
\end{align*}
holds as $A$--modules.    
\end{corollary}

We now establish a necessary condition for $A\bowtie^{f,g}(J,J')$ to be Gorenstein, which is that $J$ is a canonical module of $f(A)+J$.

\begin{proposition}\label{gorprop}
    Assume that $(A,\mathfrak{m})$ is a local Cohen--Macaulay ring, $f$ is an injective ring homomorphism and $J$ and $J'$ are finitely generated $A$--modules with $\Ann_{g(A)+J'}(J')=0$, $J\subseteq \jac(B)$ and $J'\subseteq \jac(C)$. If $A\bowtie^{f,g}(J,J')$ is Gorenstein, then $f(A)+J$ has a canonical module isomorphic to $J$.
\end{proposition}
\begin{proof}
    The Gorensteinness of $A\bowtie^{f,g}(J,J')$ implies the existence of a canonical $A\bowtie^{f,g}(J,J')$--module $\omega_{\bowtie}$ with the $A\bowtie^{f,g}(J,J')$--module isomorphism $\omega_{\bowtie}\cong A\bowtie^{f,g}(J,J')$. Injectiveness of $f$ and Theorem $\ref{dimdepthcalclemma}$ gives us $\dim A\bowtie^{f,g}(J,J')=\dim A$. Since $f$ is injective, $f(A)+J$ is an integral extension of $A$ and they have the same dimension. The classical construction of canonical modules  \cite[Theorem 21.15]{eisenbud} shows that there is a canonical module $\omega_f$ of $f(A)+J$ satisfying
    \begin{align*}
        \omega_f&\cong\ext_{A\bowtie^{f,g}(J,J')}^{\dim A\bowtie^{f,g}(J,J')-\dim f(A)+J }(f(A)+J,\omega_{\bowtie})\\
        &\cong\ext_{A\bowtie^{f,g}(J,J')}^0\left(f(A)+J,A\bowtie^{f,g}(J,J')\right)\\
&\cong \Hom_{A\bowtie^{f,g}(J,J')}\left(f(A)+J,A\bowtie^{f,g}(J,J')\right)\\
&\cong J\times0
    \end{align*}
    as an $f(A)+J$--module, where we have utilized Lemma \ref{AnnLemma}.
\end{proof}
In particular, this reduces to the result of \cite{DFF3} in the case of amalgamations.
\begin{corollary}
    Assume that $(A,\mathfrak{m})$ is a local Cohen--Macaulay ring and $J'$ is a finitely generated $A$--module with $\Ann_{g(A)+J'}(J')=0$ and $J'\subseteq \jac(C)$. If $A\bowtie^{g}J'$ is Gorenstein, then $A$ has a canonical module isomorphic to $g^{-1}(J')$.
\end{corollary}

We now characterize the Gorenstein property of $A\bowtie^{f,g}(J,J')$ with some additional assumptions. Recall that the Serre condition $(S_1)$ for a Noetherian ring $A$ is the assumption that $\depth A_{\mathfrak{p}}\ge \min(1,\height \mathfrak{p})$ for all prime ideals $\mathfrak{p}$.

\begin{theorem}\label{gorensteintheorem}
    Assume that $(A,\mathfrak{m})$ is a local ring with the proper ideal $I_0$ and $J$ and $J'$ are finitely generated $A$--modules with $\Ann_{g(A)+J'}(J')=0\neq J'$, $J\subseteq \jac(B)$ and $J'\subseteq \jac(C)$. Let $f(A)+J$ be Cohen--Macaulay and $g(A)+J'$ be $(S_1)$ and equidimensional. Then, the following are equivalent.
    \begin{enumerate}
        \item $A\bowtie^{f,g}(J,J')$ is Gorenstein,
        \item $g(A)+J'$ is Cohen--Macaulay. Furthermore, $f(A)+J$ and $g(A)+J'$ have canonical modules isomorphic to $J$ and $J'$ respectively.
    \end{enumerate}
\end{theorem}
\begin{proof}
    It is known that $A\bowtie^{f,g}(J,J')$ is given by the pull--back $A\bowtie^{f,g}(J,J')=\pi_f \times_{A/I_0}\pi_g$ with the diagram
    \[
\begin{tikzpicture}[>=Stealth]
\matrix (m) [matrix of math nodes, row sep=3em, column sep=4em] {
  f(A)+J        & A/I_0 \\
  A\bowtie^{f,g}(J,J') & g(A)+J' \\
};

\draw[-{>>}] (m-2-1) -- (m-1-1);        
\draw[-{>}] (m-1-1) -- node[above] {$\pi_f$} (m-1-2); 
\draw[-{>>}] (m-2-1) -- (m-2-2);        
\draw[-{>}] (m-2-2) -- node[right] {$\pi_g$} (m-1-2); 
\end{tikzpicture}
\]
where $\pi_f:f(A)+J\to A/I_0$ and $\pi_g:g(A)+J'\to A/I_0$ are surjective ring homomorphisms given by $\pi_f:f(a)+j\mapsto a+I_0$ and $\pi_g:g(a)+j'\mapsto a+I_0$ respectively. The theorem now directly follows from the description of Gorenstein pull--back constructions given in \cite[Theorem 4]{fiber}.
\end{proof}

We can recover the characterization of the Gorenstein property for the amalgamated algebra, amalgamated duplication and idealization as in \cite{DFF3}, \cite{Gorenstein} and \cite{Nagatagorenstein}.

\begin{corollary}
        Assume that $(A,\mathfrak{m})$ is a local ring with the proper ideal $g^{-1}(J')$, and $J'$ is a finitely generated $A$--module with $\Ann_{g(A)+J'}(J')=0\neq J'$ and $J'\subseteq \jac(C)$. Let $A$ be Cohen--Macaulay and $g(A)+J'$ be $(S_1)$ and equidimensional. Then, the following are equivalent.
    \begin{enumerate}
        \item $A\bowtie^{g}J'$ is Gorenstein,
        \item $g(A)+J'$ is Cohen--Macaulay. Furthermore, $A$ and $g(A)+J'$ have canonical modules isomorphic to $g^{-1}(J')$ and $J'$ respectively.
    \end{enumerate}
\end{corollary}
\begin{corollary}
        Assume that $(A,\mathfrak{m})$ is a local ring with a finitely generated module $M$ with $\Ann_A(M)=0$. Let $A$ be Cohen--Macaulay. Then, the following are equivalent.
    \begin{enumerate}
        \item $A\ltimes M $ is Gorenstein,
        \item $M$ is a canonical module of $A$.
    \end{enumerate}
\end{corollary}
\begin{corollary}
        Assume that $(A,\mathfrak{m})$ is a local ring with the proper, non--zero ideal $I$ with $\Ann_{A}(I)=0$. Let $A$ be Cohen--Macaulay. Then, the following are equivalent.
    \begin{enumerate}
        \item $A\bowtie I$ is Gorenstein,
        \item $A$ has a canonical module isomorphic to $I$.
    \end{enumerate}
\end{corollary}

\section{Algebroid Curves Arising from Bi--amalgamations and Amalgamations}

    We say that a ring $A$ is an algebroid curve over an algebraically closed field $k$ if it is a one-dimensional, reduced, complete local $k$--algebra with the residue field $k$. Alternatively, by Cohen's structure theorem, an algebroid curve is a one-dimensional, reduced, local ring of the form $A=k\llbracket x_1,\ldots ,x_n\rrbracket/\left( I_1 \cap \cdots\cap I_h\right)$ where $I_1,\ldots,I_h$ are prime ideals of $k\llbracket x_1,\ldots ,x_n\rrbracket$ with height $n-1$. The ring $A_i=k\llbracket x_1,\ldots ,x_n\rrbracket/I_i$ is called the $i$-th algebroid branch of $A$ and the number $h$ is called the number of branches. Algebroid curves arise naturally in the theory of curve singularities as completions of local rings of irreducible algebraic curves over $k$ according to the topology induced by their maximal ideal \cite{Gorenstein}. \medskip

    We will now define the valuation semigroup of the algebroid curve $A$ following \cite{delgado}. It can be shown that $Q(A)=k((t_1))\times \ldots \times k((t_h))$ and the integral closure of $A$ in $Q(A)$ is given by $\bar{A}=k\llbracket t_1 \rrbracket\times\ldots\times k\llbracket t_h \rrbracket$. Notice that the quotient field of $A_i$ is $k((t_i))$ and $\bar{A}_i=k\llbracket t_i\rrbracket$ is a discrete valuation ring with the valuation $v_i$, given by the order of a power series. For an element $q=(q_1,\ldots,q_h)\in Q(A)$, we define $v(q)=(v_1(q_1),\ldots,v_h(q_h))$. Let $I$ be any regular fractional ideal of $A$, then we set
    \begin{align*}
        v(I)=\{v(x):x\in I \ \text{is a regular element of } Q(A)\}.
    \end{align*}
    The set $S=v(A)\subseteq \mathbb{N}^h$ is a semigroup, called the valuation semigroup of $A$. The conductor $\zeta=\{q\in \bar{A}:q \bar{A}\subseteq A\}$ of $\bar{A}$ in $A$ is then a regular ideal of $\bar{A}$. Thus it can be written as $\left(t_1^{\delta_1}\right)\times\ldots\times \left(t_h^{\delta_h}\right)$ for a vector $\delta=(\delta_1,\ldots,\delta_h)$. The element $\delta$ is called the conductor of $S$. It is characterized as the minimal element of $S$ (with respect to the product order $\alpha \le\beta$ if $\alpha_i\le \beta_i$ for all $i=1,\ldots,h$) with the property that any vector strictly greater than $\delta$ is in $S$, i.e., $\delta+\mathbb{Z}_+^h\subseteq S$. To introduce the notion of symmetry, we need the following definition. Let $\alpha\in \mathbb{Z}^h$ be given. For each $i=1,\ldots,h$, we set
    \begin{align*}
        \Delta_i(\alpha)=\{\beta\in S:\beta_i=\alpha_i \ \text{and } \beta_j>\alpha_j \ \text{for all } j\neq i\}
    \end{align*}
    with $\Delta(\alpha)=\bigcup_{i=1}^{h}\Delta_i(\alpha)$. Finally, we say that $S$ is symmetric when $\alpha\in S$ if and only if $\Delta(\gamma-\alpha)=\emptyset$, where $\gamma=\delta-(1,\ldots,1)$. \medskip
    
    A beautiful theorem by Delgado \cite{delgado} states that the semigroup $S$ is symmetric if and only if the algebroid curve $A$ is Gorenstein. In the case of algebroid branches with $h=1$, the semigroup $S$ is numerical and the symmetry is the usual symmetry of semigroups in $\mathbb{Z}_+$, which has been extensively studied \cite{apery,herzoggen}. The number $g=\max (\mathbb{N}\setminus S)$ is called the Frobenius number and is crucial to the study of semigroups. Let us remark that $g=\gamma=\delta-1$ and the symmetry condition is given by $\alpha\in S$ if and only if $g-\alpha\notin S$. Delgado's theorem was first proved for $h=1$ by Kunz in \cite{kunz} and the general case was elaborated further in \cite{campillo}.\medskip

    All algebroid curves are Cohen--Macaulay, thus $A$ is Gorenstein if and only if it is quasi--Gorenstein. Therefore, it is natural to ask whether a canonical module always exists. A brief discussion (we refer the reader to \cite{Gorenstein} or \cite{dannacan}) shows that a canonical module, which can be identified as a fractional ideal, exists and can always be taken to satisfy $\omega_A\subseteq A$ or $A\subseteq\omega_A\subseteq \bar{A}$ for convenience. To determine when such a fractional ideal is a canonical module, we must introduce the canonical relative ideal $K$ of $S$ by
    \begin{align*}
        K=\{\alpha \in \mathbb{Z}^h: \Delta(\gamma-\alpha)=\emptyset\},
    \end{align*}
    which can be found in \cite{dannacan}. Obviously, $S$ is symmetric if and only if $K=S$. In the case $h=1$, the canonical relative ideal is given by $K=\{\alpha\in \mathbb{Z}:g-\alpha\notin S\}$. Jäger has shown in \cite{Jager} that a fractional ideal $I$ with $A\subseteq I \subseteq \bar{A}$ is a canonical ideal if and only if $v(I)=K$. Under minimal assumptions, Jäger's result also holds for $h\ge2$ (see \cite{dannacan}).\medskip

    Amalgamated duplications of algebroid curves can be used as an efficient tool for constructing Gorenstein algebroid curves. Indeed, it is shown in \cite{Gorenstein} that the amalgamated duplication of an algebroid curve along a proper regular ideal is also an algebroid curve with twice as many branches. As a consequence of the characterization of the Gorenstein property for amalgamated duplication, the duplicated algebroid curve is Gorenstein if and only if it is obtained through duplication along a canonical ideal. To the best of our knowledge, amalgamated algebroid curves have never been investigated before in the literature. In this section, we prove analogous results in the context of bi--amalgamations. Our first result is that the bi--amalgamated algebra is also an algebroid curve.

\begin{theorem}\label{algebroidthm}
    Suppose $A$ is an algebroid curve, $f:A\to B$ and $g:A\to C$ are two $k$--algebra homomorphisms into algebroid curves $f(A)+J$ and $g(A)+J'$ where $J$ and $J'$ are proper, regular ideals of $B$ and $C$ respectively, which are finitely generated $A$--modules with $f^{-1}(J)=g^{-1}(J')$, $J\subseteq \jac(B)$ and $J'\subseteq \jac(C)$. Then $A\bowtie^{f,g}(J,J')$ is also an algebroid curve with $h_{f(A)+J}+h_{g(A)+J'}$ branches. Furthermore, there is a bijective correspondence between the branches of $A\bowtie^{f,g}(J,J')$ and those of $f(A)+J$ and $g(A)+J'$ where corresponding branches are isomorphic.
\end{theorem}
\begin{proof}
    The condition $J\subseteq \jac(B)$ and $J'\subseteq \jac(C)$ with Proposition \ref{Prop.introd.}(3) implies that $A\bowtie^{f,g}(J,J')$ is local. It is furthermore reduced as it is a subring of $(f(A)+J)\times (g(A)+J')$. Let us define a $k$--algebra homomorphism $(f,g):A\to A\bowtie^{f,g}(J,J')$ by $(f,g)(a)=(f(a),g(a))$. Under the action $a\cdot(b,c)=(f(a)b,g(a)c)$, $A\bowtie^{f,g}(J,J')$ is a submodule of $(f(A)+J)\times (g(A)+J')$, which is a finite $A$--module. As $A$ is Noetherian, $A\bowtie^{f,g}(J,J')$ is also a finite $A$--module and $(f,g)$ is a finite morphism. By Theorem \ref{dimdepthcalclemma}, $\dim A\bowtie^{f,g}(J,J')=\dim A/(\ker f\cap \ker g)$. As $f(A)+J$ and $g(A)+J'$ are one-dimensional, we get $\height (\ker f)=\height (\ker g)=\height (\ker f\cap \ker g)=0$, which implies $\dim A\bowtie^{f,g}(J,J')=\dim A=1$. The maximal ideal of $A\bowtie^{f,g}(J,J')$ is $\mathfrak{m}_A\bowtie^{f,g}(J,J')$ by Proposition \ref{Prop.introd.}(3). Therefore, using Proposition \ref{Prop.introd.} \ref{isomlemma}, the residue field of $A\bowtie^{f,g}(J,J')$ is $\left(A\bowtie^{f,g}(J,J')\right)/\left(\mathfrak{m}_A\bowtie^{f,g}(J,J')\right)=A/(\mathfrak{m}_A+I_0)=k$ as $I_0\subseteq\mathfrak{m}_A$. This can alternatively be seen by $(f,g)^{-1}\left(\mathfrak{m}_A\bowtie^{f,g}(J,J')\right)=\mathfrak{m}_A$ or $\pi_{f(A)+J}^{-1}\left(\mathfrak{m}_A\bowtie^{f,g}(J,J')\right)=\mathfrak{m}_{f(A)+J}$, where $\pi_{f(A)+J}:A\bowtie^{f,g}(J,J')\to f(A)+J$ is the natural projection to the first coordinate. The completeness is easily deduced from the finite morphism $(f,g)$ with $\mathfrak{m}_A$ and $\mathfrak{m}_A\bowtie^{f,g}(J,J')$ inducing equivalent topologies. This concludes that $A\bowtie^{f,g}(J,J')$ is an algebroid curve.\medskip

    We now investigate the branching behavior of $A\bowtie^{f,g}(J,J')$. Let $L_1,\ldots,L_{h_{f(A)+J}}$ and $L_1',\ldots,L_{h_{g(A)+J'}}'$ be the minimal primes of $f(A)+J$ and $g(A)+J'$ respectively. The fact $J\nsubseteq L_i$ and $J'\nsubseteq L_{i'}'$ for every $i$ and $i'$ follows since $J$ and $J'$ are regular, $f(A)+J$ and $g(A)+J'$ are reduced, and in a reduced ring the set of zero divisors is the union of minimal primes. We claim that $\overline{L_1},\ldots,\overline{L_{h_{f(A)+J}}},\overline{L_1'},\ldots,\overline{L_{h_{g(A)+J'}}'}$ are the minimal primes of $A\bowtie^{f,g}(J,J')$. It can be checked easily that $\overline{L_i}\neq \overline{L_{i'}'}$ from $J\subseteq \overline{L_{i'}'}$. As $J\nsubseteq L_i$, the conditions of Proposition \ref{Prop.introd.} \ref{primelemma} apply and we obtain $\left(A\bowtie^{f,g}(J,J')\right)_{\overline{L_i}}=(f(A)+J)_{L_i}$. Therefore, the minimality of $L_i$ implies the minimality of $\overline{L_i}$. A similar argument shows that the ideals $\overline{L_{i'}'}$ are minimal. Now let $\mathfrak{P}$ be a minimal prime ideal of $A\bowtie^{f,g}(J,J')$. By Proposition \ref{Prop.introd.} \ref{primelemma}, $\mathfrak{P}$ is of the form $\overline{L}$ with $L\in \spec(f(A)+J)\cup\spec(g(A)+J')$. It is obvious that non--minimality of $L$ contradicts the minimality of $\mathfrak{P}$. This shows that $A\bowtie^{f,g}(J,J')$ has $h_{f(A)+J}+h_{g(A)+J'}$ branches. It is easy to see that the projections $\overline{\pi}_{f(A)+J,i}:A\bowtie^{f,g}(J,J')\to (f(A)+J)/L_i, \overline{\pi}_{f(A)+J,i}(b,c)=b+L_i$ and $\overline{\pi}_{g(A)+J',i'}:A\bowtie^{f,g}(J,J')\to (g(A)+J')/L_{i'}', \overline{\pi}_{g(A)+J',i'}(b,c)=c+L_{i'}'$ have kernels $\overline{L_i}$ and $\overline{L_{i'}'}$. Thus they induce isomorphisms $A\bowtie^{f,g}(J,J')/\overline{L_i}\cong (f(A)+J)/L_i$ and $A\bowtie^{f,g}(J,J')/\overline{L_{i'}'}\cong (g(A)+J')/L_{i'}'$. This completes the proof.
\end{proof}

As a corollary, we get the following novel result on amalgamated algebroid curves and the result of \cite{Gorenstein} on duplicated algebroid curves.

\begin{corollary}
    Suppose $A$ is an algebroid curve, $f:A\to B$ is a $k$--algebra homomorphism into the algebroid curve $f(A)+J$ where $J$ is a proper, regular ideal of $B$, which is a finitely generated $A$--module with $J\subseteq \jac(B)$. Then $A\bowtie^{f}J$ is also an algebroid curve with $h_A+h_{f(A)+J}$ branches. Furthermore, there is a bijective correspondence between the branches of $A\bowtie^{f}J$ and those of $A$ and $f(A)+J$ where corresponding branches are isomorphic.
\end{corollary}
\begin{corollary}
    Suppose $A$ is an algebroid curve and $I$ is a proper, regular ideal of $A$. Then $A\bowtie I$ is also an algebroid curve with $2h_A$ branches. Furthermore, there is a one-to-two correspondence between the branches of $A\bowtie I$ and those of $A$ where corresponding branches are isomorphic. 
\end{corollary}

    Instead of following the above algebraic approach, one can give an explicit representation of the algebroid curve $A\bowtie^{f,g}(J,J')$ as a quotient of a power series ring. Indeed, let $A=k\llbracket x_1,\ldots ,x_n\rrbracket/I_A$ be the quotient representation of $A$ and let us define $\varphi_{B}:k\llbracket x_1,\ldots ,x_n\rrbracket\to \bar{B}$ as the composition of the natural projection $\pi_A:k\llbracket x_1,\ldots ,x_n\rrbracket\to A$, the morphism $f:A\to B$, and the embedding $B\hookrightarrow\bar{B}$. Similarly, define $\varphi_C:k\llbracket x_1,\ldots ,x_n\rrbracket\to \bar{C}$ as the composition of the natural projection $\pi_A:k\llbracket x_1,\ldots ,x_n\rrbracket\to A$, the morphism $g:A\to C$, and the embedding $C\hookrightarrow\bar{C}$. Clearly, $\ker \varphi_B=\pi_A^{-1}(\ker f)=\ker f^{L}+I_Ak\llbracket x_1,\ldots ,x_n\rrbracket$ and $\ker \varphi_C=\pi_A^{-1}(\ker g)=\ker g^{L}+I_Ak\llbracket x_1,\ldots ,x_n\rrbracket$ where $\ker f^{L}$ and $\ker g^{L}$ are ideals generated by some fixed liftings of minimal generating sets of $\ker f$ and $\ker g$. Take liftings $i_1,\ldots,i_m\in k\llbracket x_1,\ldots ,x_n\rrbracket$ such that the projections $\overline{i_1},\ldots, \overline{i_m}\in A$ form a minimal generating set of $I_0$. We shall identify $J$ and $J'$ with their images in $\bar{B}$ and $\bar{C}$.\medskip

    Set $y_1,\ldots,y_m$ and $z_1,\ldots,z_m$ to be new indeterminates. Define two functions
    \begin{align*}
        \psi_B:k\llbracket x_1,\ldots ,x_n,y_1,\ldots,y_m,z_1,\ldots,z_m\rrbracket&\to \bar{B}; \\
        \psi_B(x_j)&=\varphi_B(x_j),\\
        \psi_B(y_{j'})&=0, \\
        \psi_B(z_{j'})&=\varphi_B(i_{j'}),
    \end{align*}
    and
    \begin{align*}
        \psi_C:k\llbracket x_1,\ldots ,x_n,y_1,\ldots,y_m,z_1,\ldots,z_m\rrbracket&\to \bar{C}; \\
        \psi_C(x_j)&=\varphi_C(x_j),\\
        \psi_C(y_{j'})&=\varphi_C(i_{j'}), \\
        \psi_C(z_{j'})&=0,
    \end{align*}
    for all $j=1,\ldots,n$ and $j'=1,\ldots,m$, and then extend them to be ring homomorphisms. Furthermore, let us define the ring homomorphism
    \begin{align*}
        \Omega:k\llbracket x_1,\ldots ,x_n,y_1,\ldots,y_m,z_1,\ldots,z_m\rrbracket\to \bar{B}\times \bar{C}; \qquad \Omega(F)=(\psi_B(F),\psi_C(F)).
    \end{align*}
    We claim that $\operatorname{im} \Omega=A\bowtie^{f,g}(J,J')$. Given any $F\in k\llbracket x_1,\ldots ,x_n,y_1,\ldots,y_m,z_1,\ldots,z_m\rrbracket$, we can decompose it uniquely as
    \begin{align*}
        F&=F_x(x_1,\ldots, x_n)+F_{x,y}(x_1,\ldots, x_n,y_1,\ldots, y_m)+F_{x,z}(x_1,\ldots, x_n,z_1,\ldots, z_m)\\
        &+F_{x,y,z}(x_1,\ldots, x_n,y_1,\ldots, y_m,z_1,\ldots,z_m).
    \end{align*} 
    It follows from the definition that $\psi_B(F_{x,y})=\psi_B(F_{x,y,z})=\psi_C(F_{x,z})=\psi_C(F_{x,y,z})=0$ and $\psi_B(F_x)=\varphi_B(F_x)=f(\overline{F_x})$, $\psi_C(F_x)=\varphi_C(F_x)=g(\overline{F_x})$. Therefore, we obtain $\Omega(F)=(f(\overline{F_x})+\psi_B(F_{x,z}),g(\overline{F_x})+\psi_C(F_{x,y}))$ with $\psi_B(F_{x,z})\in (\varphi_B(i_1),\ldots,\varphi_B(i_m))=f(I_0)=J$ and $\psi_C(F_{x,y})\in (\varphi_C(i_1),\ldots,\varphi_C(i_m))=g(I_0)=J'$. This shows that $\operatorname{im} \Omega\subseteq A\bowtie^{f,g}(J,J')$. Conversely, it is easy to see that any $(f(a)+j,g(a)+j')\in A\bowtie^{f,g}(J,J')$ is of the form $\Omega(F)$. This shows that $\operatorname{im} \Omega=A\bowtie^{f,g}(J,J')$. Therefore, we have $\operatorname{im} \Omega=A\bowtie^{f,g}(J,J')=k\llbracket x_1,\ldots ,x_n,y_1,\ldots,y_m,z_1,\ldots,z_m\rrbracket/\ker \Omega$. Furthermore, $\ker\Omega=\ker\psi_B\cap\ker\psi_C$ and
    \begin{align*}
\ker\psi_B&=(y_1,\ldots,y_m)+(z_1-i_1,\ldots,z_m-i_m)+\ker f^L+I_Ak\llbracket x_1,\ldots ,x_n,y_1,\ldots,y_m,z_1,\ldots,z_m\rrbracket\\
\ker\psi_C&=(y_1-i_1,\ldots,y_m-i_m)+(z_1,\ldots,z_m)+\ker g^L+I_Ak\llbracket x_1,\ldots ,x_n,y_1,\ldots,y_m,z_1,\ldots,z_m\rrbracket.
    \end{align*}
    Now it can be checked that $\ker\psi_B$ and $\ker\psi_C$ are respectively intersections of $h_B$ and $h_C$ prime ideals with heights $n+2m-1$. This gives us the explicit representation of $A\bowtie^{f,g}(J,J')$ as a quotient.\medskip

Now we can give the second main result of this section.

\begin{theorem}\label{algebroidgorthm}
    The algebroid curve $A\bowtie^{f,g}(J,J')$ is Gorenstein if and only if $J$ and $J'$ are isomorphic to canonical ideals of $f(A)+J$ and $g(A)+J'$ respectively. 
\end{theorem}
\begin{proof}
    The proof follows from the characterization Theorem \ref{gorensteintheorem} as all algebroid curves are Noetherian and Cohen--Macaulay, thus satisfy the Serre condition $(S_1)$ and are equidimensional.
\end{proof}
Similarly, we get the following new result on amalgamations and the result of \cite{Gorenstein}.

\begin{corollary}\label{algebroidgorcor}
    The algebroid curve $A\bowtie^{f}J$ is Gorenstein if and only if $J$ is isomorphic to the canonical ideal of $f(A)+J$.
\end{corollary}
\begin{corollary}
        The algebroid curve $A\bowtie I$ is Gorenstein if and only if $I$ is isomorphic to the canonical ideal of $A$.
\end{corollary}

\begin{example}
    Let us consider the bi--amalgamated algebroid curve $A\bowtie^{f,g}(J,J')$ with the data $A=k\left\llbracket t \right\rrbracket$ and
    \begin{gather*}
        B=k\left\llbracket t^4,t^7,t^9  \right\rrbracket,\qquad J=\left( t^4,t^9 \right),\qquad f:A\to B, f(t)=t^7\\
        C=k\left\llbracket t^5,t^8,t^{11}  \right\rrbracket, \qquad J'=\left( t^5,t^8 \right),\qquad g:A\to C, g(t)=t^{11}.
    \end{gather*}
    Then we have $f^{-1}(J)=\left(t^3\right)=g^{-1}(J')$ and $f(A)+J=B, g(A)+J'=C$. The value semigroups of $B$ and $C$ are respectively 
    \begin{align*}
        S_B=\langle4,7,9\rangle&=\{0,4,7,8,9,11,12,13,\ldots\}\\
        S_C=\langle5,8,11\rangle&=\{0,5,8,10,11,13,15,16,18,19,20,\ldots\}
    \end{align*}
    with Frobenius numbers $g_B=10, g_C=17$ and canonical ideals
    \begin{align*}
        K(S_B)&=\{0,4,5,7,8,9,11,\ldots\}\\
        K(S_C)&=\{0,3,5,8,10,11,13,14,15,16,18,\ldots\}.
    \end{align*}
    This shows that $I_B=\left(1,t^5\right)$ and $I_C=\left(1,t^3\right)$ are canonical fractional ideals of $B$ and $C$ respectively, as they satisfy $v(I_B)=K(S_B)$ and $v(I_C)=K(S_C)$. Therefore, $J$ and $J'$ are canonical ideals of $B$ and $C$ respectively, since $J=t^4 I_B$ and $J'=t^5 I_C$. The conditions of Theorem \ref{algebroidgorthm} are satisfied, and the algebroid curve $A\bowtie^{f,g}(J,J')$ is Gorenstein.

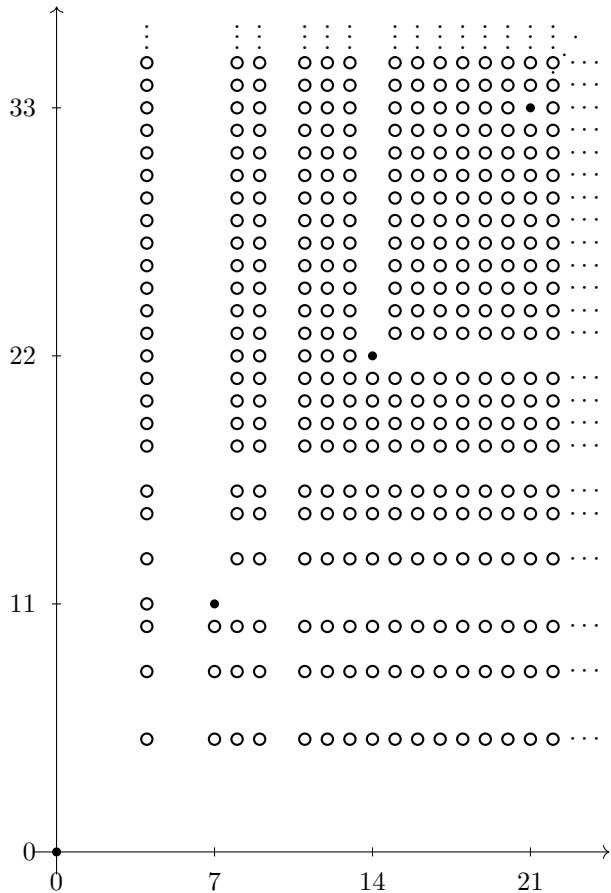
\begin{figure}[ht]
\centering
\begin{tikzpicture}[scale=0.3] 
    \draw[->] (-1,0) -- (24.5,0);
    \draw[->] (0,-1) -- (0,37.5);

    \foreach \x in {0, 7, 14, 21} {
        \draw (\x, -0.2) -- (\x, 0.2);
        \node[below] at (\x, -0.5) {\x};
    }
    \foreach \y in {0, 11, 22, 33} {
        \draw (-0.2, \y) -- (0.2, \y);
        \node[left] at (-0.5, \y) {\y};
    }

    \foreach \x/\y in {0/0, 7/11, 14/22, 21/33} {
        \fill[black] (\x, \y) circle (0.2);
    }

    \foreach \x in {4, 8, 9, 11, 12, 13, 15, 16, 17, 18, 19, 20, 21, 22} {
        \foreach \y in {5, 8, 10, 13, 15, 16, 18, 19, 20, 21, 23, 24, 25, 26, 27, 28, 29, 30, 31, 32, 33, 34, 35} {
            \def\isdiag{0}
            \ifnum\x=21 \ifnum\y=33 \def\isdiag{1} \fi \fi
            \ifnum\isdiag=0
                \draw[thick] (\x, \y) circle (0.25);
            \fi
        }
    }

    \foreach \y in {5, 8, 10} { \draw[thick] (7, \y) circle (0.25); }
    \foreach \y in {5, 8, 10, 13, 15, 16, 18, 19, 20, 21} { \draw[thick] (14, \y) circle (0.25); }

    \foreach \x in {4} { \draw[thick] (\x, 11) circle (0.25); }
    \foreach \x in {4, 8, 9, 11, 12, 13} { \draw[thick] (\x, 22) circle (0.25); }

    \fill (22, 34.57) circle (0.06);
    \fill (22.5, 35.35) circle (0.06);
    \fill (23, 36.14) circle (0.06);

    \foreach \y in {5, 8, 10, 13, 15, 16, 18, 19, 20, 21, 23, 24, 25, 26, 27, 28, 29, 30, 31, 32, 33, 34, 35} {
        \node at (23.5, \y) {$\cdots$};
    }
    \foreach \x in {4, 8, 9, 11, 12, 13, 15, 16, 17, 18, 19, 20, 21, 22} {
        \node at (\x, 36.5) {$\vdots$};
    }

\end{tikzpicture}
\caption{The value semigroup of the bi--amalgamated algebroid curve $A \bowtie^{f,g} (J, J')$.}
\end{figure}
    
\end{example}

\begin{example}
    Let us consider the amalgamated algebroid curve $A\bowtie^f J$ with the data $A=k\left\llbracket t \right\rrbracket$ and 
    \begin{align*}
        B=k\left\llbracket t^3,t^7,t^8\right\rrbracket, \qquad J=\left(t^7,t^8\right), \qquad f:A\to B, f(t)=t^3.
    \end{align*}
    Then we have $f(A)+J=B$. The value semigroup of $B$ is given by
    \begin{align*}
        S=\langle3,7,8\rangle=\{0,3,6,7,8,9,10,\ldots\}
    \end{align*}
    with the Frobenius number $g=5$ and the canonical ideal
    \begin{align*}
        K(S)=\{0,1,3,4,6,7,\ldots\}.
    \end{align*}
    This shows that $I=(1,t)$ is a canonical fractional ideal of $B$ as $v(I)=K(S)$. Therefore, $J$ is a canonical ideal of $B$ since $J=t^7 I$. The conditions of Corollary \ref{algebroidgorcor} are satisfied, and the algebroid curve $A\bowtie^f J$ is Gorenstein.
    
\end{example}

\begin{figure}[ht]
\centering
\begin{tikzpicture}[scale=0.35] 
    \draw[->] (-1,0) -- (24.5,0);
    \draw[->] (0,-1) -- (0,24.5);

    \foreach \x in {0, 3, 6, 9, 12, 15, 18, 21} {
        \draw (\x, -0.2) -- (\x, 0.2);
        \node[below] at (\x, -0.5) {\x};
        \draw (-0.2, \x) -- (0.2, \x);
        \node[left] at (-0.5, \x) {\x};
    }

    \foreach \x in {0, 3, 6, 9, 12, 15, 18, 21} {
        \fill[black] (\x, \x) circle (0.2);
    }

    
    \foreach \y in {7, 8, 10, 11} {
        \draw[thick] (12, \y) circle (0.25);
    }

    \foreach \x in {15, 18, 21} {
        \foreach \y in {7, 8, 10, 11, 13, 14, 15, 16, 17, 18, 19, 20, 21, 22, 23} {
            \ifnum\x=\y 
            \else
                \draw[thick] (\x, \y) circle (0.25);
            \fi
        }
    }

    \fill (22.2, 22.2) circle (0.06);
    \fill (22.6, 22.6) circle (0.06);
    \fill (23.0, 23.0) circle (0.06);

    \foreach \x in {15, 18, 21} {
        \node at (\x, 24) {$\vdots$};
    }
    
    \foreach \y in {7, 8, 10, 11, 13, 14, 15, 16, 17, 18, 19, 20, 21, 22, 23} {
        \node at (23.5, \y) {$\cdots$};
    }
\end{tikzpicture}
\caption{The value semigroup of the amalgamated algebroid curve $A \bowtie^f J$.}
\end{figure}
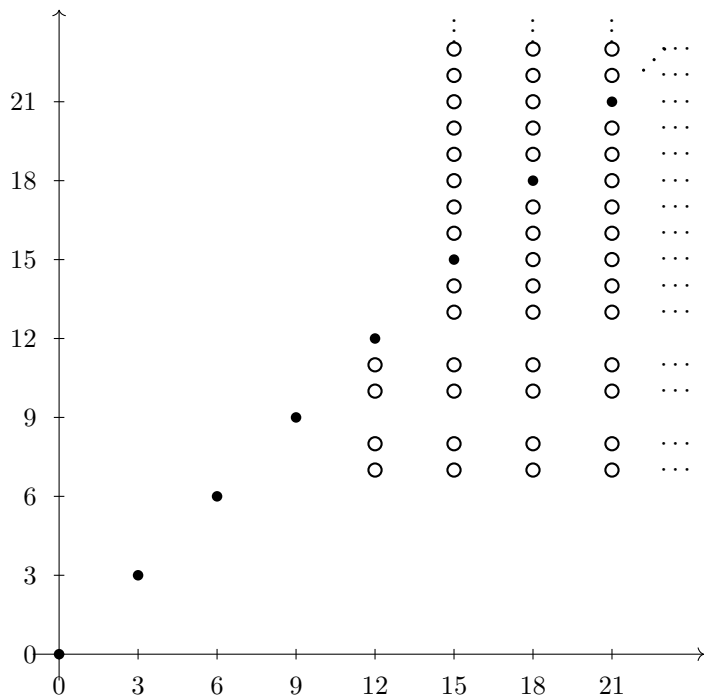

\section{A remark on the universally catenary property}

Bi--amalgamations naturally generalize some classic pull--back constructions of polynomial and power series rings. In the previous section, we have also seen an application of the bi--amalgamated algebra construction to curve singularities. In this section, we describe the bi--amalgamated algebra as a homomorphic image of a polynomial ring. The results here generalize those of \cite{iranlilar}.

\begin{proposition}\label{homimageprop}
    Assume that $J$ and $J'$ are generated by $\{j_\mu\}_{\mu\in M}$ and $\{j'_\nu\}_{\nu\in N}$ respectively as $A$--modules. Then $A\bowtie^{f,g}(J,J')$ is a homomorphic image of $A[\{x_\mu\}_{\mu\in M}, \{x'_\nu\}_{\nu\in N}]$, where $x_\mu$ and $x'_\nu$ are indeterminates.
\end{proposition}
\begin{proof}
    Define a homomorphism $\varphi:A[\{x_\mu\}_{\mu\in M}, \{x'_\nu\}_{\nu\in N}]\to A\bowtie^{f,g}(J,J')$ by $\varphi(a)=(f(a),g(a))$, $\varphi(x_\mu)=(j_\mu,0)$ and $\varphi(x'_\nu)=(0,j'_\nu)$. It is obvious that $\varphi$ is surjective and thus $A\bowtie^{f,g}(J,J')$ is a homomorphic image of $A[\{x_\mu\}_{\mu\in M}, \{x'_\nu\}_{\nu\in N}]$.
\end{proof}
\begin{corollary}
    Assume that $J$ is generated by $\{j_\mu\}_{\mu\in M}$ as an $A$--module. Then $A\bowtie^{f}J$ is a homomorphic image of $A[\{x_\mu\}_{\mu\in M}]$, where $x_\mu$ are indeterminates.
\end{corollary}
\begin{corollary}
    Assume that $M$ is generated by $\{m_\mu\}_{\mu\in M}$ as an $A$--module. Then $A\ltimes M$ is a homomorphic image of $A[\{x_\mu\}_{\mu\in M}]$, where $x_\mu$ are indeterminates.
\end{corollary}
\begin{corollary}
    Assume that $I$ is generated by $\{i_\mu\}_{\mu\in M}$ as an ideal of $A$. Then $A\bowtie I$ is a homomorphic image of $A[\{x_\mu\}_{\mu\in M}]$, where $x_\mu$ are indeterminates.
\end{corollary}

As a consequence of the above proposition, we reobtain a special case of the Noetherian characterization of $A\bowtie^{f,g}(J,J')$. Indeed, we get that whenever $J$ and $J'$ are finitely generated $A$--modules, $A\bowtie^{f,g}(J,J')$ is Noetherian if and only if $A$ is Noetherian. This agrees with the results already found in the literature (see \cite{KLT}) and gives an easier proof in the finitely generated case.\medskip

We say that a Noetherian ring $A$ is catenary if all saturated chains of prime ideals between any two prime ideals have the same length. If all finitely generated $A$--algebras are catenary, we call $A$ universally catenary. From the above description, we have the following characterization of the universally catenary property for the bi--amalgamated algebra.

\begin{proposition}\label{catenaryprop}
    Assume $J$ and $J'$ are finitely generated as $A$--modules. Then $A\bowtie^{f,g}(J,J')$ is universally catenary if and only if $A/I_0$ is universally catenary.
\end{proposition}
In particular, we get the following characterization for amalgamated algebras as given in \cite{iranlilar}.
\begin{corollary}
        Assume $J$ is a finitely generated $A$--module. Then $A\bowtie^{f}J$ is universally catenary if and only if $A$ is universally catenary.
\end{corollary}
\begin{corollary}
        Assume $M$ is a finitely generated $A$--module. Then $A\ltimes M$ is universally catenary if and only if $A$ is universally catenary.
\end{corollary}
\begin{corollary}
        Assume $I$ is an ideal of $A$. Then $A\bowtie I$ is universally catenary if and only if $A$ is universally catenary.
\end{corollary}

 \end{document}